\documentclass[preprint,12pt]{elsarticle}
\usepackage{amssymb}
\usepackage{amsmath}
\usepackage{amsthm}
\usepackage{mathrsfs}
\usepackage{booktabs}
\usepackage{diagbox}
\usepackage[colorlinks=true, linkcolor=blue, citecolor=blue, urlcolor=blue]{hyperref}

\newtheorem{example}{Example}%
\newtheorem{theorem}{Theorem}%

\usepackage{geometry}
\journal{arXiv}

\begin{document}

\begin{frontmatter}



\title{KANtrol: A Physics-Informed Kolmogorov-Arnold Network Framework for Solving Multi-Dimensional and Fractional Optimal Control Problems}


\author[af1]{Alireza Afzal Aghaei} 

\affiliation[af1]{organization={Independent Researcher},
            state={Isfahan},
            country={Iran}}

\begin{abstract}
In this paper, we introduce the KANtrol framework, which utilizes Kolmogorov-Arnold Networks (KANs) to solve optimal control problems involving continuous time variables. We explain how Gaussian quadrature can be employed to approximate the integral parts within the problem, particularly for integro-differential state equations. We also demonstrate how automatic differentiation is utilized to compute exact derivatives for integer-order dynamics, while for fractional derivatives of non-integer order, we employ matrix-vector product discretization within the KAN framework. We tackle multi-dimensional problems, including the optimal control of a 2D heat partial differential equation. The results of our simulations, which cover both forward and parameter identification problems, show that the KANtrol framework outperforms classical MLPs in terms of accuracy and efficiency.
\end{abstract}


\end{frontmatter}

\section{Introduction}

In recent years, Physics-Informed Neural Networks (PINNs) have emerged as a powerful tool for solving differential and integral equations, as well as other physical systems, by incorporating domain knowledge directly into the learning process \cite{raissi2019physics, aghaei2024pinnies}. These models integrate the governing physical laws, such as conservation principles, thermodynamic relations, or Maxwell's equations, into the architecture of deep neural networks \cite{nellikkath2022physics,khan2022physics,arzani2021uncovering}. Unlike traditional data-driven approaches, PINNs leverage this prior knowledge to ensure that the learned solutions adhere to the physical constraints of the system, thereby improving both accuracy and generalization, especially in cases with limited or noisy data.

Building upon the foundational principles of PINNs, several extensions have been proposed to further enhance the model's capabilities through advanced deep learning techniques. One such extension involves the use of Deep Operator Networks (DeepONets), which aim to learn operators that map between infinite-dimensional function spaces, providing a more flexible framework for handling complex input-output relationships in physical systems \cite{lu2019deeponet}. Other methods, such as Fourier Neural Operators, utilize spectral methods to efficiently model multi-scale phenomena, particularly in turbulent fluid flows and weather prediction \cite{li2020fourier}. Additionally, Variational Physics-Informed Neural Networks enhance solution accuracy by combining classical variational principles with deep learning, optimizing energy-based loss functions \cite{kharazmi2019variational}.

Another notable advancement is the development of Kolmogorov-Arnold Networks (KANs), integrated with physics-informed techniques. KANs represent a novel and powerful advancement in solving complex problems, excelling at learning non-linear mappings between high-dimensional spaces. These networks are especially effective in capturing intricate dynamics and decomposing functions into simpler components, making them ideal for modeling both local and global behaviors in physical systems. ordinary differential equations \cite{aghaei2024fkan, aghaei2024rkan}, partial differential equations \cite{liu2024kan, wang2024kolmogorov, rigas2024adaptive, shuai2024physics, ranasinghe2024ginn}, dynamical systems \cite{patra2024physics}, and fractional delay differential equations \cite{aghaei2024fkan}. 

In certain engineering applications, the complexity of systems demands more sophisticated PINN models. This has led to the development of Auxiliary PINNs and the PINNIES framework for solving integral equations \cite{yuan2022pinn, aghaei2024pinnies}. These models extend the traditional PINN framework by incorporating additional neural networks or constraints to handle integral forms of physical laws, making them more suitable for complex systems where standard differential formulations may not suffice. Additionally, Physics-Informed Convolutional Neural Networks have been developed to operate on signal data, making them well-suited for tasks such as time-series forecasting and signal reconstruction \cite{shi2024physics, russell2022physics}. Another significant advancement is the application of PINNs in optimization tasks, particularly in solving control system problems, as demonstrated in recent studies focused on control, design, and resource management optimization \cite{seo2024solving}.

Control systems are essential in various engineering domains, from robotics and aerospace to industrial automation and energy systems. These systems manage the behavior of dynamic processes by adjusting inputs to achieve desired outputs while maintaining stability, efficiency, and performance. The mathematical foundation for designing these systems lies in control theory, which enables the precise manipulation of system dynamics, even in the presence of disturbances and uncertainties \cite{lewis2012optimal}. As the complexity of control systems increases, researchers have focused on developing efficient methods to solve control problems. PINNs have been successfully applied to control systems by incorporating the governing physical laws of dynamic systems directly into the learning process.

The use of neural networks to solve optimal control problems has garnered significant attention in recent decades. Various neural network architectures have been explored depending on the nature of the control problem, be it continuous, discrete, or fractional, along with the corresponding integration techniques. Early works, such as \cite{becerikli2003intelligent}, proposed neural network models based on Multilayer Perceptrons (MLPs) for continuous-time ordinal equations, utilizing the Runge-Kutta-Butcher method for the numerical solution of the system. Later, \cite{liu2012neural} extended this approach to discrete-time systems, also using MLPs. More advanced methods have integrated Hamiltonian principles into the neural network training process. For example, \cite{effati2013optimal} employed MLPs with a Hamiltonian-based approach for solving continuous-time ordinal equations, a method later extended by \cite{ghasemi2017nonlinear} to handle fractional and nonlinear equations, showcasing the flexibility of MLPs. Similarly, \cite{sabouri2017neural} applied Simpson's Rule to solve fractional differential equations, combining traditional numerical integration methods with neural networks. 

The application of neural networks for solving fractional control equations has expanded in recent years. \cite{yavari2019fractional} addressed fractional infinite-horizon equations using MLPs with a Hamiltonian approach, while \cite{mortezaee2020infinite} focused on continuous infinite-horizon ordinal equations. Delay differential equations have also been investigated by \cite{kheyrinataj2020fractional}, who employed MLPs combined with Hamiltonian techniques. This approach was later revisited in \cite{kheyrinataj2023solving}, where Simpson’s Rule was used to address similar problems. In more recent studies, deep learning models have been employed for continuous-time ordinal equations. For instance, \cite{benning2021deep} demonstrated the effectiveness of Hamiltonian methods in deep learning architectures for solving control tasks, and this trend has continued with works such as \cite{bottcher2022ai}, \cite{barry2022physics}, and \cite{yin2024aonn}. These models address continuous-time ordinal and partial differential equations, continuing the use of Hamiltonian methods to solve control problems. Additionally, \cite{na2024physics} employed Euler-Lagrange principles in their approach. 

Recurrent Neural Networks (RNNs) have also been applied to optimal control problems, particularly for discrete-time ordinal equations, as shown by \cite{chen2018optimal}, emphasizing their ability to handle time-dependent data. The combination of neural networks with traditional numerical methods, such as in the work by \cite{sanchez2018real}, highlights the utility of numerical integration techniques for solving continuous-time ordinal control problems. More recently, \cite{yin2023optimal} and \cite{aghaei2024pinnies} employed Gaussian Quadrature for continuous-time ordinal equations, while \cite{mowlavi2023optimal} applied the Midpoint Rule for partial differential equations. 

Despite the advancements in using neural networks for control tasks, no prior work has specifically addressed optimal control problems using KANs. In this paper, we propose the development of Kolmogorov-Arnold Networks for the accurate prediction of a class of optimal control problems, including those with ODE/PDE constraints, fractional derivatives, and integro-differential equations (IDEs). To achieve this, we employ Gaussian quadrature to approximate the integral terms that arise in the functional or constraints of the problem and introduce the \textit{KANtrol} framework. Our contributions are as follows:
\begin{itemize}
    \item We develop the \textit{KANtrol} framework, which leverages KANs for efficiently solving optimal control problems.
    \item A method is explained for approximating integrals in physics-informed tasks when solving integral operator problems.
    \item We address fractional and multi-dimensional optimal control problems.
    \item Parameter identification of optimal control problems is also considered in this work.
\end{itemize}

The rest of the paper is organized as follows: Section 2 provides an overview of Kolmogorov-Arnold Networks. In Section 3, we present the proposed methodology for solving optimal control problems using KANs. Section 4 includes several numerical examples to demonstrate the effectiveness of the proposed approach. Finally, in Section 5, we conclude with a discussion on the limitations of this work and suggest potential future research directions.

\section{Kolmogorov-Arnold Network}
Kolmogorov-Arnold Networks (KANs) are a class of neural networks grounded in the Kolmogorov-Arnold representation theorem, which proves that any multivariate continuous function can be decomposed into a composition of univariate functions and addition operations \cite{liu2024kan,aghaei2024fkan}. Mathematically, they can be represented as:
\[
\psi\left(x_{1}, \ldots, x_{d}\right) = \sum_{q=1}^{2d+1} \Phi_{q}\left(\sum_{p=1}^{d} \phi_{q, p}\left(\zeta_{p}\right)\right),
\]
where the functions \(\phi_{q, p}\) are univariate and map each input \(\zeta_p\), such that \(\phi_{q, p}: [0, 1] \to \mathbb{R}\), while \(\Phi_q: \mathbb{R} \to \mathbb{R}\) handles further transformations. In KANs, a KAN layer is defined by the collection of univariate functions, mathematically denoted as:
\[
\boldsymbol{\Phi} = \{\phi_{q, p}\}, \quad p=1, 2, \ldots, n_{\text{in}}, \quad q=1, 2, \ldots, n_{\text{out}}.
\]
The architecture of KANs can be visualized as a deep, multi-layer structure where each layer consists of neurons performing linear transformations followed by non-linear activations. For an input vector \(\boldsymbol{\zeta}\), the output of the network is computed via a series of compositions, written as:
\[
\operatorname{KAN}(\boldsymbol{\zeta}) = \left(\boldsymbol{\Phi}_{L-1} \circ \boldsymbol{\Phi}_{L-2} \circ \cdots \circ \boldsymbol{\Phi}_{0}\right)\boldsymbol{\zeta},
\]
where \(L\) denotes the total number of layers in the network.

One of the distinguishing features of KANs is their use of specialized activation functions. These include B-splines, which are particularly useful for flexible, non-linear interpolation. The activation function in a KAN is represented as:
\[
\phi(\zeta) = w_{b} b(\zeta) + w_{s} \text{spline}(\zeta),
\]
where \(b(\zeta) = \frac{\zeta}{1 + \exp({-\zeta})}\) is a sigmoid-like function such as SiLU, \(\text{spline}(\zeta) = \sum_{i=0}^k c_{i} B_{i}(\zeta)\) is a combination of B-spline basis functions \(B_i(\zeta)\), polynomials of a specified degree, with the parameters \(c_i\), \(w_b\), and \(w_s\) being trainable.

\section{Methodology}

Optimal control problems involve determining a control function that optimizes a performance criterion while adhering to dynamic system constraints. These systems are often modeled by ordinary or partial differential equations (ODEs/PDEs) that describe the evolution of the system's state over time. Mathematically, an optimal control problem can be expressed as the minimization (or maximization) of a cost functional $\mathfrak{J}(\psi)$, which depends on both the control $\psi(\tau)$ and the states $\boldsymbol\xi(\tau)$. The objective is to find an optimal control function $\psi^*(\tau)$ that minimizes the functional:
\begin{equation}
    \mathfrak{J}(\psi) = \int_\Xi \mathcal{L}(\boldsymbol\xi(\tau), \psi(\tau), \tau) \, d\tau + \mathcal{G}(\boldsymbol\xi(\tau_f)),
    \label{eq:optcont}
\end{equation}
where $\Xi = [{\tau_0},{\tau_f}]^d$ is the $d$-dimensional problem domain, $ \mathcal{L}(\boldsymbol\xi(\tau), \psi(\tau), \tau)$ represents the running cost, and $\mathcal{G}(\boldsymbol\xi(\tau_f))$ is a terminal cost function. The state $\boldsymbol\xi(\tau)\in\mathbb{R}^{\mathfrak{n}}$ evolves according to a set of differential equations, typically expressed as:
\[
\mathscr{D}\big[{\boldsymbol\xi}(\tau)\big] = \mathcal{F}(\boldsymbol\xi(\tau), \psi(\tau), \tau),
\]
where $\mathscr{D}[\cdot]$ is a differential operator, subject to boundary or initial conditions at the times $\tau_0$ and $\tau_f$.

In various applications, additional constraints may be imposed on both the control and the state variables, making the problem more complex. These can include control bounds, state limitations, infinite time horizons, or integral constraints on the system. In more advanced scenarios, fractional derivatives or integro-differential equations may describe the system's dynamics, necessitating specialized numerical approximation techniques. 

To solve optimal control problems, various mathematical approaches are employed, such as Pontryagin's Maximum Principle and the Hamilton-Jacobi-Bellman equation \cite{lewis2012optimal}. In this work, we leverage Kolmogorov-Arnold networks to approximate both the control function and the state evolution by embedding the problem's physics directly into the neural network architecture. This approach enables us to efficiently handle the wide range of constraints and dynamics found in optimal control problems, including fractional and integral terms. In this method, each unknown function in the problem is represented using a KAN architecture as follows:
\begin{align*}
    \psi(\tau) &= \mathrm{KAN}_0(\tau), \\
    \boldsymbol\xi_i(\tau) &= \mathrm{KAN}_i(\tau), \quad i = 1, 2, \dots, \mathfrak{n}.
\end{align*}
Using these approximations, we then calculate the cost functional. For this purpose, numerical integration techniques are employed to obtain a precise approximation of the integral. Among popular methods such as Simpson's rule, Midpoint Rule, Newton-Cotes, Monte Carlo, and Gaussian quadrature, the latter has proven to be the most accurate approximation method \cite{aghaei2024pinnies}. Applying Gaussian Quadrature, we approximate the cost functional as follows:
\begin{equation}
    \mathfrak{J} \approx \frac{1}{2} [\tau_f - \tau_0] \sum_{j=1}^\mathfrak{Q} \omega_j \mathcal{L}(\boldsymbol\xi(\hat\tau_j), \psi(\hat\tau_j), \hat\tau_j),
    \label{eq:costapprox}
\end{equation}
where $\hat\tau_j = \frac{\tau_f - \tau_0}{2} \hat\tau + \frac{\tau_0 + \tau_f}{2}$, and $\tau_j$ are the roots of the Legendre polynomials \cite{aghaei2024pinnies,firoozsalari2023deepfdenet}. The parameter $\mathfrak{Q}$ represents the order of quadrature, which theoretically improves accuracy with larger values. However, due to Runge's phenomenon, achieving higher-order accuracy may not always be possible.

The constraints of this optimization problem can be divided into three main categories: integer-order derivatives, fractional-order derivatives, and integro-differential terms. In the simplest case, the constraints involve only derivatives of the state variable with respect to time. For such cases, automatic differentiation techniques in deep learning frameworks can be employed to compute these derivatives efficiently \cite{raissi2019physics, aghaei2024fkan, aghaei2024pinnies}. However, for more complex cases involving fractional derivatives or integro-differential equations, it becomes necessary to either extend the KAN framework with auxiliary variables \cite{yuan2022pinn}, or apply a general approximation method, as described below.

\paragraph{Fractional Order Derivatives}
Fractional calculus is a generalization of traditional calculus, where the concept of derivatives and integrals is extended to non-integer (fractional) orders. This extension allows for more accurate modeling of physical phenomena that exhibit memory effects, long-range interactions, or anomalous diffusion, which are often inadequately captured by classical integer-order calculus. In fractional calculus, the derivative of a function can be taken to a non-integer order, denoted as $\mathscr{D}^\alpha$, where $\alpha$ is a real or complex number.  Fractional derivatives are particularly useful in fields such as viscoelasticity, control theory, and fluid dynamics, where the system's behavior depends not only on its current state but also on its historical states \cite{aghaei2024pinnies,yavari2019fractional,ma2023bi,pang2019fpinns}. These derivatives enable the modeling of such systems by accounting for this inherent memory. For example, in control systems, fractional derivatives provide more flexibility in characterizing system dynamics by introducing additional degrees of freedom, resulting in more accurate and realistic models \cite{yavari2019fractional,kheyrinataj2020fractional}.

An interesting feature of fractional derivatives is the non-uniqueness of their definitions. In fact, several types of fractional derivatives are commonly employed, with the most prominent being the Riemann-Liouville, Caputo, and Grünwald-Letnikov derivatives \cite{baleanu2012fractional,firoozsalari2023deepfdenet}. Each of these has distinct advantages, making them suitable for different applications, depending on the problem's nature. Among these, the Caputo derivative is particularly favored, as it extends classical integer-order derivatives. Importantly, as the fractional order $\alpha$ approaches an integer, the Caputo derivative converges to the corresponding integer-order derivative. This property makes it especially useful in scenarios where a smooth transition between fractional and integer-order calculus is needed.

In this paper, we consider the Caputo fractional derivative, which is defined as:
\begin{equation*}
    \mathscr{D}^\alpha \xi(\tau) = \frac{1}{\Gamma(1-\alpha)} \int_{0}^{\tau} \frac{\frac{d}{d\tau}\xi(\eta)}{(\tau-\eta)^\alpha} d\eta,
\end{equation*}
where $\alpha \in (0,1)$ is the fractional order, and $\Gamma(\cdot)$ is the Gamma function \cite{aghaei2024fkan,firoozsalari2023deepfdenet}. As can be seen, computing this derivative requires evaluating a singular integral, which may not be feasible or accurate using standard numerical integration techniques. To address this challenge, recent works have proposed the use of finite difference discretization to approximate the derivative \cite{taheri2024accelerating, aghaei2024pinnies, pang2019fpinns}. In this approach, the fractional derivative is approximated by evaluating the neural network at a set of discrete points, which can either be equidistant \cite{taheri2024accelerating} or arbitrarily spaced \cite{aghaei2024pinnies}. In the following, we recall the approach proposed in \cite{aghaei2024pinnies} for approximating the Caputo fractional derivative operator. 
\begin{theorem}
\label{thm:operationalmatrix}
The Caputo fractional derivative can be approximated using the matrix-vector product:
\begin{equation*}
    \mathscr{D}^\alpha\big[\tilde\xi\, \big] \approx \mathcal{D}^\alpha \mathbf{\tilde\xi},
\end{equation*}
where $\tilde\xi_j = \xi(\tau_j)$, $\tau_j$ for $j=1,\ldots, M$ are some points in the problem domain, and the lower triangular matrix $\mathcal{D}^\alpha$ is defined as:
\begin{equation*}
    \mathcal{D}^\alpha = \begin{bmatrix}
0 \\
\mu_0^{(1)} & \mu_1^{(1)}\\
\mu_0^{(2)} & \mu_1^{(2)} & \mu_2^{(2)}\\
&\vdots\\
\mu_0^{(M-2)} & \mu_1^{(M-2)} & \mu_2^{(M-2)} & \mu_3^{(M-2)} & \cdots & \mu_{M-2}^{(M-2)}\\
\mu_0^{(M-1)} & \mu_1^{(M-1)} & \mu_2^{(M-1)} & \mu_3^{(M-1)} & \cdots & \mu_{M-2}^{(M-1)} & \mu_{M-1}^{(M-1)}
\end{bmatrix},
\end{equation*}
where $\mu_k^{(i)} = {\left[\lambda_k^{(i)} - \lambda_{k-1}^{(i)}\right]}/{\Gamma(2-\alpha)}$ and:
\begin{equation*}
\lambda_k^{(i)} = \frac{(\tau_i-\tau_k)^{1-\alpha}-(\tau_i-\tau_{k+1})^{1-\alpha}}{\tau_{k+1}-\tau_k},
\end{equation*}
for any $0\le k < M$.
\end{theorem}
\begin{proof}
    See \cite{aghaei2024pinnies, taheri2024accelerating}.
\end{proof}
\paragraph{Integro-Differential Equations}
Integro-differential equations are a class of equations that combine aspects of both integral and differential equations \cite{aghaei2024pinnies}. These equations involve an unknown function that appears under both integral and differential operators, making them useful for modeling complex systems where the current state depends on both the history (through the integral) and the instantaneous rate of change (through the derivative). IDEs naturally arise in various fields such as biology, physics, engineering, and economics, where processes are governed by both local dynamics and cumulative effects. Examples include population dynamics with delayed feedback, heat conduction with memory effects, and systems with hereditary properties \cite{lakshmikantham1995theory}.

Mathematically, an integro-differential equation may take the form:
\[
\mathscr{D}\big[\xi(\tau)\big] = \mathcal{F}\left(\tau, \xi(\tau), \int_{\tau_0}^{\tau} \mathcal{K}(\tau, \iota) \xi(\iota) d\iota \right),
\]
where the unknown function $\xi(\tau)$ is both differentiated and integrated, with the integral term accounting for cumulative effects over time. The inclusion of the integral term adds complexity to solving these equations, as one must compute or approximate this term. Numerical methods, such as discretization techniques and specialized quadrature methods, are often employed to approximate this operator \cite{aghaei2024pinnies}. In this work, we employ Gauss-Legendre quadrature to compute the integral terms, as previously described for cost approximation. This can be expressed mathematically as:
\begin{equation}
\int_{\tau_0}^{\tau} \mathcal{K}(\tau, \iota) \xi(\iota) d\iota \approx \frac{1}{2}[\tau - \tau_0] * \left[\mathbf{K} * \tilde\xi\right] \boldsymbol\omega,
\label{eq:int}
\end{equation}
where $*$ denotes element-wise multiplication, $\mathbf{K}$ is a matrix that evaluates the kernel function $\mathcal{K}(\tau, \iota)$ over the problem domain $\tau$ and the integration domain $\iota$, $\boldsymbol{\omega}$ is a column vector containing Gauss-Legendre weights, and $\tilde{\xi}$ is a matrix of neural network predictions on a grid of the problem and integration domains \cite{aghaei2024pinnies}.

Following the approximation of the integral term using this approach (Equation \eqref{eq:int}), computing fractional derivatives based on theorem \ref{thm:operationalmatrix}, and calculating integer-order derivatives via automatic differentiation and computing the functional of the problem using Equation \eqref{eq:costapprox}, the physics-informed loss function of the network can be constructed as follows:
\begin{equation}
   \begin{aligned}
        \mathrm{Loss}(\boldsymbol{\tau}) =  \ \upsilon_\mathfrak{J} \mathfrak{J} 
    &+ \upsilon_{\mathscr{R}} \sum_{\xi \in \boldsymbol{\xi}}\sum_{\tau' \in \boldsymbol{\tau}} \big[\mathscr{R}_\xi({\tau'})\big]^2 \\
    & + \upsilon_{\mathscr{B}} \sum_{\xi \in \boldsymbol{\xi}}\sum_{(\tau', \xi_{\tau'}) \in {\mathbf{\mathscr{B}}}} \big[\xi(\tau') - \xi_{\tau'}\big]^2 \\
    & + \upsilon_{\mathscr{O}} \sum_{\xi \in \boldsymbol{\xi}}\sum_{(\tau',\mathcal{o}) \in \mathscr{O}} \big[\xi(\tau')-\mathcal{o}\big]^2,
   \end{aligned}
    \tag{KANtrol}
\end{equation}
where $\mathscr{R}_\xi(\tau')$ represents the residual of the approximated state equation $\xi(\cdot)$ using KAN architecture on input data $\tau'$, $\mathscr{B}$ denotes the set of initial or boundary conditions, and $\mathscr{O}$ represents set of observable data. The weights $\upsilon_{\mathfrak{J}}, \upsilon_{\mathscr{R}}, \upsilon_{\mathscr{B}}, \upsilon_{\mathscr{O}}$ are regularization hyperparameters that balance the contributions of each term to the overall loss. These hyperparameters are tuned to ensure the network optimally satisfies both the physical constraints and problem requirements.

In Figure \ref{fig:schema}, we illustrate the overall structure of the proposed KANtrol framework. In the subsequent section, we apply this architecture to simulate a series of numerical examples for optimal control problems, demonstrating the effectiveness and validity of our method.

\begin{figure}
    \centering
    \includegraphics[width=0.97\linewidth]{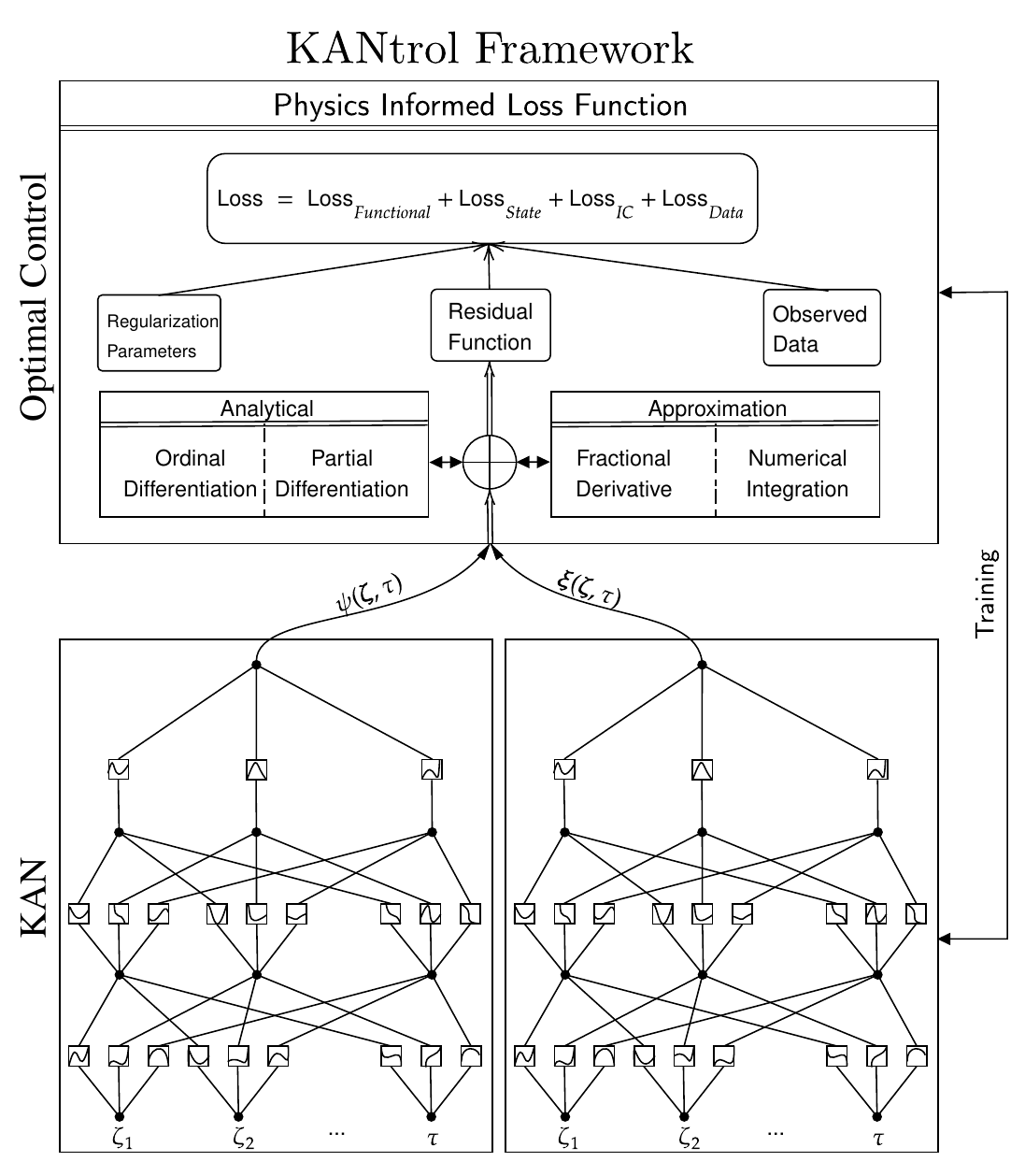}
    \caption{Overall Structure of KANtrol framework for solving optimal control problems using Kolmogorov-Arnold Networks.}
    \label{fig:schema}
\end{figure}

\section{Numerical Experiments}
In this section, we simulate a variety of optimal control problems, encompassing both forward and inverse problems. For the forward problems, the control problem is solved with predefined initial or boundary conditions \cite{kafash2012application}. In contrast, for the inverse problems, we omit the initial conditions and assume the presence of unknown parameters that must be identified, relying instead on noisy observational data \cite{mombaur2010human,jean2018inverse,levine2012continuous,dvijotham2010inverse,menner2019constrained}.

We implemented the method using the official KAN implementation in Python, built with PyTorch. For comparison, we employed the official implementations of rational KAN (rKAN) \cite{aghaei2024rkan}, fractional KAN (fKAN) \cite{aghaei2024fkan}, and classical MLPs within the PINNIES framework \cite{aghaei2024pinnies}. In all examples, we used the same network architecture, $[1,10,1]$, and the L-BFGS optimizer. For KAN, we set the number of grid intervals to $5$. For fKAN, we reported the best results across different polynomial degrees ranging from $2$ to $6$, while for rKAN, we reported the best Pade hyperparameters, with $[p/q]$ for $p=2,3,\ldots,6$ and $q=p+1,\ldots,6$. All experiments were conducted on a personal computer equipped with an Intel Core-i3 10100 processor and 16GB of RAM, running the EndeavourOS Linux distribution.
\begin{example}
\label{ex:frac-forward}
For the first example, we consider the problem \eqref{eq:optcont} with 
$$\mathcal{L}(\boldsymbol\xi(\tau), \psi(\tau), \tau)  = \left( \xi_1(\tau) - 1 - \tau^\frac{3}{2} \right)^2 + \left( \xi_2(\tau) - \tau^\frac{5}{2} \right)^2 + \left( \psi(\tau) - \frac{3\sqrt{\pi}}{4} \tau + \tau^\frac{5}{2} \right)^2,$$
$ \mathscr{D}^{0.5} [\xi_1(\tau)] = \xi_2(\tau) + u(\tau)$, $\mathscr{D}^{0.5} [\xi_2(\tau)] = \xi_1(\tau) + \frac{15\sqrt{\pi}}{16} \tau^2 - \tau^\frac{3}{2} - 1$, and the initial conditions $\xi_1(0) = 1$, $\xi_2(0) = 0$, with $\Xi=[\tau_0, \tau_f] = [0,1]$. The exact solution to this problem is given by $\xi_1(\tau) = 1 + \tau^{\frac{3}{2}}$, $\xi_2(\tau) = \tau^{\frac{5}{2}}$, and $\psi(\tau) = \frac{3\sqrt{\pi}}{4} \tau - \tau^{\frac{5}{2}}$. The analytical solution to this problem yields a cost value $\mathfrak{J} = 0$ in the optimal case.

By setting $M = 2000$ in Theorem \ref{thm:operationalmatrix} and $\upsilon_{\mathfrak{J}} = 10^{-1}$, we trained three different neural network architectures to approximate $\psi(\tau), \xi_1(\tau)$, and $\xi_2(\tau)$. The results of this simulation are shown in Table \ref{tbl:frac-forward}, where the proposed method demonstrates superior performance in terms of mean absolute error between the predicted and analytical solutions. However, based on our experiments and previous findings \cite{liu2024kan}, this network is slightly slower compared to the others.

\begin{table}[ht]
\centering
\begin{tabular}{@{}ccccc@{}}
\toprule
\backslashbox{Criteria}{Model} & \begin{tabular}[c]{@{}c@{}}KAN\\ $k = 3$\end{tabular} & \begin{tabular}[c]{@{}c@{}}Fractional KAN\\ $k=2$\end{tabular} & \begin{tabular}[c]{@{}c@{}}Rational KAN\\ Pade [3/6]\end{tabular} & \begin{tabular}[c]{@{}c@{}}PINNIES\\ MLP\end{tabular} \\ \midrule
$\mathfrak{J}$   & {$\mathbf{7.71 \times 10^{-8}}$}  & $5.95 \times 10^{-5}$ & $5.02 \times 10^{-6}$ & $3.03 \times 10^{-6}$ \\
${\psi}(\tau)$   & {$\mathbf{2.28 \times 10^{-4}}$}  & $6.02 \times 10^{-3}$ & $1.64 \times 10^{-3}$ & $1.18 \times 10^{-3}$ \\
${\xi_1}(\tau)$  & {$\mathbf{4.04 \times 10^{-5}}$}  & $1.91 \times 10^{-3}$ & $4.35 \times 10^{-4}$ & $2.97 \times 10^{-4}$ \\
${\xi_2}(\tau)$  & {$\mathbf{1.44 \times 10^{-5}}$}  & $4.75 \times 10^{-4}$ & $1.49 \times 10^{-4}$ & $2.04 \times 10^{-4}$ \\ \bottomrule
\end{tabular}
\caption{Comparison of mean absolute error (MAE) of KANtrol with similar neural network architectures for solving fractional optimal control problem \ref{ex:frac-forward}.}
\label{tbl:frac-forward}
\end{table}

\end{example}

\begin{example}
    \label{ex:frac-inverse}
In this example, we revisit Example \ref{ex:frac-forward}, but in the context of an inverse problem. Here, we omit the initial conditions and assume that noisy observational data for the state variables are provided. To introduce additional complexity, we also assume that the term $\frac{15\sqrt{\pi}}{16}$ in the first constraint is unknown, allowing the network to identify this parameter, denoted as $\kappa$. We assume that 20 random data points are available in the problem domain, with noisy targets generated by adding noise, sampled from a standard Gaussian distribution with a noise factor of 0.05, to the exact solution. The simulation results are shown in Figure \ref{fig:inverse}. In this case, the simulation achieved $\mathfrak{J} = 2.06\times10^{-6}$, and the MAE for the control and state variables was $6.41\times10^{-4}, 5.28\times10^{-4}, 9.11\times10^{-4}$, respectively. Figure \ref{fig:kappa} also illustrates the network's convergence in identifying the unknown value $\kappa$.

    \begin{figure}[ht]
        \centering
        \includegraphics[height=0.63\textheight]{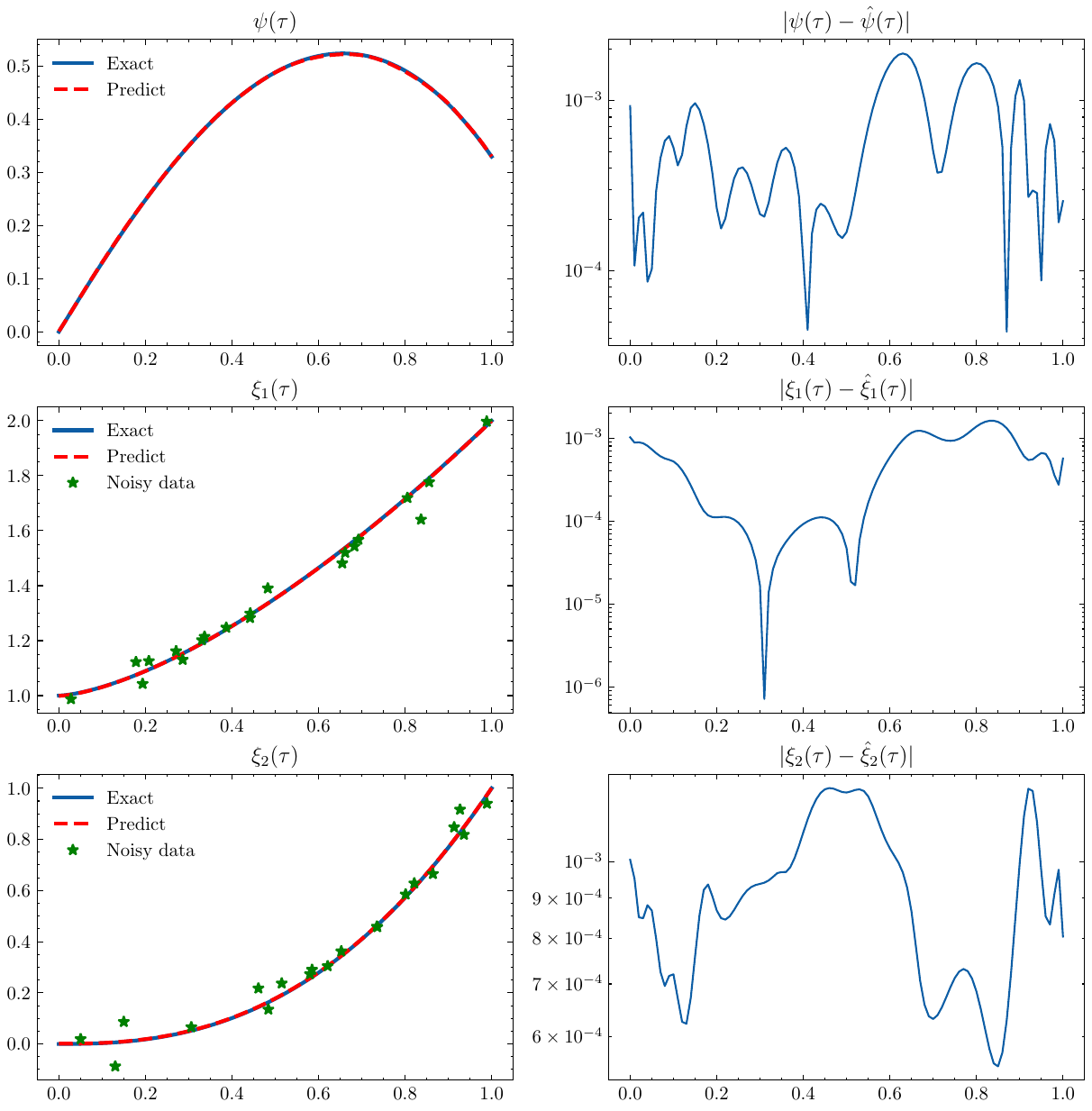}
        \caption{Parameter identification of optimal control problems using the KANtrol framework, for Example \ref{ex:frac-inverse}.}
        \label{fig:inverse}
    \end{figure}

    \begin{figure}[ht]
        \centering
        \includegraphics[width=0.8\linewidth]{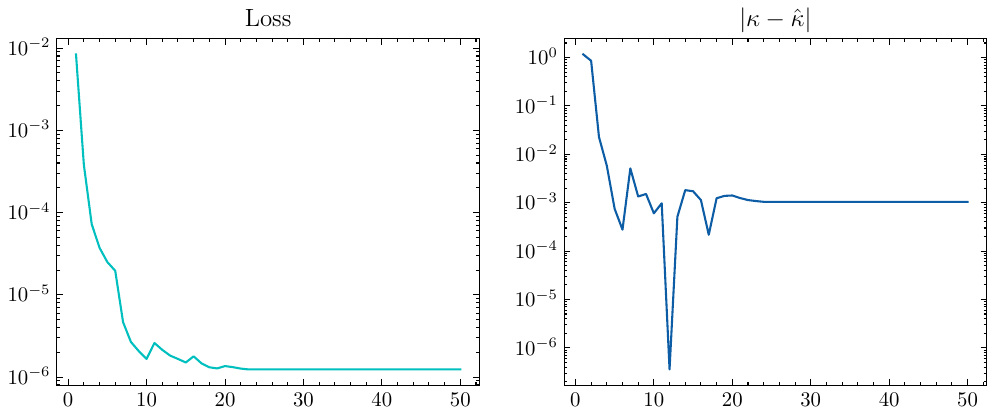}
        \caption{The training loss and convergence of the KANtrol framework towards the exact $\kappa$ value during the training phase are depicted. The y-axis of the plots is on a logarithmic scale.}
        \label{fig:kappa}
    \end{figure}
\end{example}

\begin{example}
    \label{ex:opt-ide}
For the next experiment, we consider \eqref{eq:optcont} as 
$$\mathcal{L}(\xi(\tau), \psi(\tau), \tau)  = \left( \xi(\tau) - \exp({\tau^2}) \right)^2 + \left( \psi(\tau) - 2\tau - 1 \right)^2,$$
over $\Xi=[\tau_0, \tau_f] = [0,1]$ with an integro-differential equation constraint \cite{heydari2023orthonormal,aghaei2024pinnies}:
\begin{equation*}
    \mathfrak{D}[\xi(\tau)] = \psi(\tau) - \xi(\tau) + \int_0^\tau (2\tau^2 + \tau) \exp({{\iota\tau-\iota^2}}) \xi(\iota) d\iota,
\end{equation*}
given initial condition $\xi(0)=1$. The exact solution to this problem is \(\xi(\tau) = \exp(\tau^2)\) and \(\psi(\tau) = 2\tau + 1\), resulting in an optimal value of \(\mathfrak{J} = 0\). For the simulation, we trained KAN, fKAN, rKAN, and an MLP network with 100 training points, using \( \upsilon_{\mathfrak{J}} = 10^{-3} \) and setting the other hyperparameters to one. The results of this neural network simulation are presented in Table \ref{tbl:opt-integro}. As observed, the superiority of the KAN-based architecture over the plain MLPs is evident.

\begin{table}[ht]
\centering
\begin{tabular}{@{}ccccc@{}}
\toprule
\backslashbox{Criteria}{Model} & \begin{tabular}[c]{@{}c@{}}KAN\\ $k = 3$\end{tabular} & \begin{tabular}[c]{@{}c@{}}Fractional KAN\\ $k=2$\end{tabular} & \begin{tabular}[c]{@{}c@{}}Rational KAN\\ Pade [3/6]\end{tabular} & \begin{tabular}[c]{@{}c@{}}PINNIES\\ MLP\end{tabular} \\ \midrule
$\mathfrak{J}$   & {$\mathbf{2.20 \times 10^{-6}}$}  & $5.49 \times 10^{-5}$ & $5.56 \times 10^{-6}$ & $1.44 \times 10^{-3}$ \\
${\psi}(\tau)$   & {$\mathbf{1.05 \times 10^{-3}}$}  & $6.34 \times 10^{-3}$ & $1.96 \times 10^{-3}$ & $3.33 \times 10^{-2}$ \\
${\xi}(\tau)$  & {$\mathbf{1.29 \times 10^{-4}}$}  & $8.64 \times 10^{-4}$ & $1.83 \times 10^{-3}$ & $5.40 \times 10^{-3}$ \\ \bottomrule
\end{tabular}
\caption{Comparison of the mean absolute error across various neural network architectures for solving the integro-differential optimal control problem \ref{ex:opt-ide}.}
\label{tbl:opt-integro}
\end{table}

\end{example}

\begin{example}
    \label{ex:opt-2d}
For a more complex optimal control problem, consider the following two-dimensional problem \cite{hassani2020generalized}:
$$
\mathcal{L}(\xi(\zeta,\tau), \psi(\zeta,\tau), \zeta,\tau)  = \left[ \xi(\zeta,\tau) - \tau^4 \sin(\zeta) \right]^2 + \left[ \psi(\zeta,\tau) - \tau^3 \cos(\zeta) \right]^2,
$$
over $\Xi = [0,1]^2$ with respect to PDE:
\begin{equation}
\begin{aligned}
    \frac{\partial \xi(\zeta, \tau)}{\partial \tau} &= \cos(\xi(\zeta, \tau)) 
    + 2 \sin(\zeta) \frac{\partial \xi(\zeta, \tau)}{\partial \zeta} 
    + \frac{\partial^2 \xi(\zeta, \tau)}{\partial \zeta^2} \\
    &\quad + 6 \sin(\zeta) u(\zeta, \tau) 
    - \cos(\tau^4 \sin(\zeta)) \\
    &\quad - \tau^3 \left(\tau \sin(2\zeta) - \tau \sin(\zeta) + 3 \sin(2\zeta)\right) \\
    &\quad + 4 \sin(\zeta) \tau^3,
\end{aligned}
\end{equation}
with initial and boundary conditions $\xi(\zeta,0) = \xi(0,\tau)=0$. The exact solution to this problem is \( \xi(\zeta,\tau) = \tau^4 \sin(\zeta) \) and \( \psi(\zeta,\tau) = \tau^3 \cos(\zeta) \), which yields an analytical objective value of \( \mathcal{J} = 0 \). We employed a nested Gauss-Legendre method to approximate the cost functional and used automatic differentiation to compute the partial derivatives of the unknown solution with respect to \( \tau \). Given that the number of training points increases quadratically with the grid approach, we set \(M = 25 \) training points for each dimension in this example. The results of the neural network simulations for this problem using $\upsilon_\mathfrak{J}=10^{-6}$ are presented in Table \ref{tbl:opt-2d}.

\begin{table}[ht]
\centering
\begin{tabular}{@{}ccccc@{}}
\toprule
\backslashbox{Criteria}{Model} & \begin{tabular}[c]{@{}c@{}}KAN\\ $k = 3$\end{tabular} & \begin{tabular}[c]{@{}c@{}}Fractional KAN\\ $k = 2$\end{tabular} & \begin{tabular}[c]{@{}c@{}}Rational KAN\\ Pade [3/6]\end{tabular} & \begin{tabular}[c]{@{}c@{}}PINNIES\\ MLP\end{tabular} \\ \midrule
$\mathfrak{J}$   & {$\mathbf{2.41 \times 10^{-7}}$}  & $1.94 \times 10^{-5}$ & $5.56 \times 10^{-6}$ & $6.49 \times 10^{-6}$ \\
${\psi}(\tau)$   & {$\mathbf{1.27 \times 10^{-4}}$}  & $3.46 \times 10^{-3}$ & $1.96 \times 10^{-3}$ & $2.57 \times 10^{-3}$ \\
${\xi}(\tau)$  & {$\mathbf{4.14 \times 10^{-4}}$}  & $3.76 \times 10^{-3}$ & $1.83 \times 10^{-3}$ & $1.10 \times 10^{-3}$ \\ \bottomrule
\end{tabular}
\caption{Comparison of the MAE for various neural network architectures in solving the multi-dimensional optimal control problem \ref{ex:opt-2d}.}
\label{tbl:opt-2d}
\end{table}

\end{example}

\begin{example}
    \label{ex:opt-3d}
    For the last experiment, we consider optimal control of two-dimensional heat PDE given by
\begin{align*}
        \mathcal{L}(\xi(\boldsymbol\zeta,\tau), \psi(\boldsymbol\zeta,\tau), \boldsymbol\zeta,\tau)  &= \left[ \xi(\zeta_1,\zeta_2,\tau) - \sin(\zeta_1) \sin(\zeta_2) \exp(-\tau) \right]^2 \\
        &+ \left[ \psi(\zeta_1,\zeta_1,\tau) - \sqrt{\exp(\zeta_1+\zeta_2-\tau)} \right]^2,
    \end{align*}
over $\Xi=[0,\pi]^3$ with respect to $\frac{\partial \xi}{\partial \tau} = \Delta \xi + \psi + \chi$ where $\Delta$ denotes Laplacian operator, $\chi(\boldsymbol{\zeta},\tau)$ is a known source term, with initial condition $\xi(\zeta_1,\zeta_2,0) = \sin(\zeta_1)\sin(\zeta_2)$ and four boundary conditions on domain $[0,\pi]$. The exact solution to this problem is given by \(\xi(\zeta_1,\zeta_2,\tau) = \sin(\zeta_1) \sin(\zeta_2) \exp(-\tau)\) and \(\psi(\zeta_1,\zeta_2,\tau) = \sqrt{\exp(\zeta_1,\zeta_2,\tau)}\) \cite{kouatchou2001finite}, with the optimal objective value being \(\mathfrak{J} = 0\). We simulated this problem using 10 training points in each dimension with KAN, rKAN, fKAN, and MLPs, setting \(\upsilon_\mathfrak{J} = 10^0\). The comparison results are shown in Table \ref{tbl:opt-3d}, and Figure \ref{fig:3d} illustrates the predicted solution using KAN as the base model.
\begin{table}[ht]
\begin{tabular}{@{}ccccc@{}}
\toprule
\backslashbox{Criteria}{Model} & \begin{tabular}[c]{@{}c@{}}KAN\\ $k = 3$\end{tabular} & \begin{tabular}[c]{@{}c@{}}Fractional KAN\\ $k = 2$\end{tabular} & \begin{tabular}[c]{@{}c@{}}Rational KAN\\ Pade [3/6]\end{tabular} & \begin{tabular}[c]{@{}c@{}}PINNIES\\ MLP\end{tabular} \\ \midrule
$\mathfrak{J}$   & {$\mathbf{1.12 \times 10^{-2}}$}  & $2.08 \times 10^{-2}$ & $5.94 \times 10^{-2}$ & $1.61 \times 10^{-2}$ \\
${\psi}(\tau)$   & {$\mathbf{2.99 \times 10^{-4}}$}  & $4.85 \times 10^{-3}$ & $1.51 \times 10^{-3}$ & $3.61 \times 10^{-4}$ \\
${\xi}(\tau)$    & {$\mathbf{2.19 \times 10^{-3}}$}  & $2.80 \times 10^{-3}$ & $2.89 \times 10^{-3}$ & $2.66 \times 10^{-3}$ \\ \bottomrule
\end{tabular}
\caption{Comparison of the mean absolute error across various neural network architectures for solving the multi-dimensional optimal control problem of heat equation \ref{ex:opt-3d}.}
\label{tbl:opt-3d}
\end{table}

\begin{figure}[ht]
    \centering
    \includegraphics[width=0.95\linewidth]{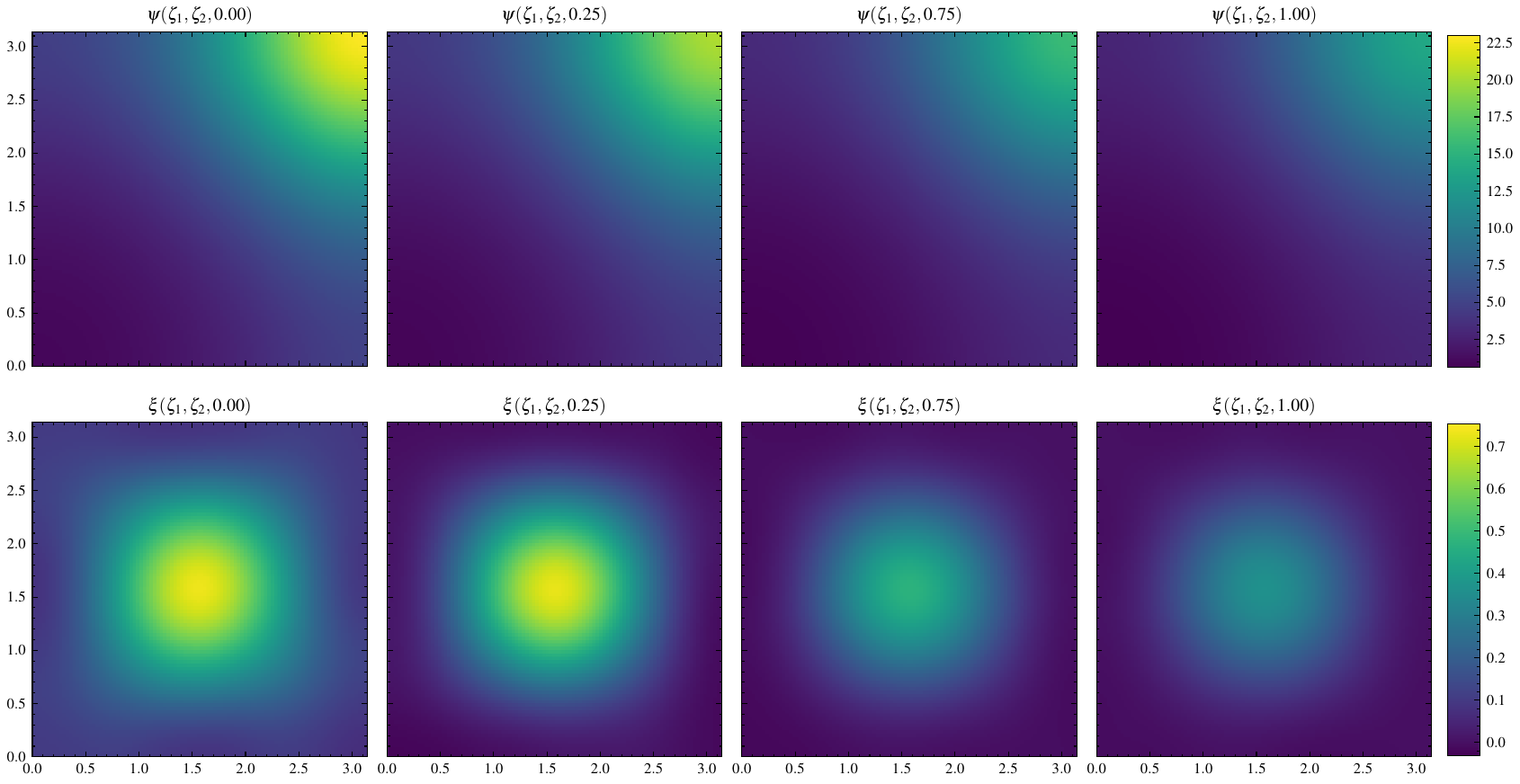}
    \caption{Simulation results for the optimal control of a 2D heat PDE over the domain \([0, \pi]^2\) at different time values up to one. The top row displays the control variable, while the bottom row illustrates the evolution of heat as the state variable.}
    \label{fig:3d}
\end{figure}
\end{example}

\section{Conclusion}
In this paper, we introduced the KANtrol framework, a physics-informed neural network approach for solving optimal control problems using Kolmogorov-Arnold networks. We demonstrated how the integration component can be approximated using Gaussian quadrature, applied to both the functional of the optimal control and the state equations expressed as integro-differential equations. Additionally, we illustrated how operational matrices for fractional derivatives can be employed to accurately approximate fractional Caputo derivatives within the KANtrol architecture, using only matrix-vector product operations.

Through a series of experiments on well-established optimal control problems, including forward and parameter identification problems, as well as the optimal control of partial differential equations, we evaluated the performance of the KANtrol framework. We also simulated an optimal control problem involving a two-dimensional heat equation. Our results were compared with various state-of-the-art approaches, such as rational KAN \cite{aghaei2024rkan}, fractional KAN \cite{aghaei2024fkan}, and the PINNIES framework \cite{aghaei2024pinnies}. In all cases, the KAN architecture provided superior numerical approximations; however, it proved slower compared to multi-layer networks like those used in the PINNIES framework. A noted limitation is the necessity to optimize the $\upsilon_\mathfrak{J}$ hyperparameter, which balances the tradeoff between the accuracy of the running cost function and system dynamics. Future work could focus on accelerating KANs or employing Hamiltonian techniques to eliminate the need for this hyperparameter.

\bibliographystyle{elsarticle-num} 
\bibliography{refs}

\end{document}